\documentclass[11pt]{amsart}
\usepackage{amsmath}
\usepackage{amsfonts}
\usepackage{amssymb}
\usepackage{graphicx}
\usepackage{amsthm}
\usepackage{xcolor,cancel}
\usepackage[normalem]{ulem}
\def\p{\partial}
\def\Riem{{\mathcal R}}
\def\R{\mathbb{R}}
\def\Vol{\mathrm{Vol}}
\def\Ric{\mathrm{Ric}}

\newtheorem*{thmA}{Theorem A}
\newtheorem*{thmB}{Theorem B}

\theoremstyle{definition}

\newtheorem*{remark}{Remark}
\usepackage{fullpage}

\usepackage{hyperref}   
\usepackage{enumitem}
\begin{document}

\title{Evolution of relative Yamabe constant \\ under Ricci
  Flow}
\author{Boris Botvinnik}
\address{
Department of Mathematics\\
University of Oregon \\
Eugene, OR, 97403\\
USA
}
\email{botvinn@uoregon.edu}
\author{Peng Lu}
\address{
Department of Mathematics\\
University of Oregon \\
Eugene, OR, 97403\\
USA
}
\email{penglu@uoregon.edu}
\subjclass[2010]{53C27, 57R65, 58J05, 58J50}

  \date{\today}
\maketitle
\begin{abstract}
Let $W$ be a manifold with boundary $M$ given together with a
conformal class $\bar C$ which restricts to a conformal class $C$ on
$M$.  Then the relative Yamabe constant $Y_{\bar C}(W,M;C)$ is
well-defined.  We study the short-time behavior of the relative Yamabe
constant $Y_{[\bar g_t]}(W,M;C)$ under the Ricci flow $\bar g_t$ on
$W$ with boundary conditions that mean curvature $H_{\bar g_t}\equiv
0$ and $\bar{g}_t|_M\in C = [\bar{g}_0]$.  In particular, we show that
if the initial metric $\bar{g}_0$ is a Yamabe metric, then, under some
natural assumptions, $\left.\frac{d}{dt}\right|_{t=0}Y_{[\bar
    g_t]}(W,M;C)\geq 0$ and is equal to zero if and only the metric
$\bar{g}_0$ is Einstein.
\end{abstract}  

\section{Introduction}
In this short note, we analyze the behavior of the relative Yamabe
constant of manifolds with boundary under the Ricci flow with
appropriate boundary conditions. It turns out that the evolution
equation of the relative Yamabe constant is analogous to the one on
closed manifolds, given in \cite{Lu}.

\subsection{Relative Yamabe constant}
Let $W$ be a compact manifold with boundary $M =\partial W\neq
\emptyset$ and $\dim W = n\geq 3$. For a metric $\bar g$ on $W$ we
denote by $g=\bar g|_{M}$ and by $H_{\bar g}$ the mean curvature along
the boundary $M$ with respect to the outward unit normal vector.  We
also denote by $[\bar g]$ and $[g]$ the corresponding conformal
classes, and by $\mathcal{C}(W)$ and $\mathcal{C}(M)$ the spaces of
conformal classes on $W$ and $M$, respectively. Given $\bar C\in
\mathcal{C}(W)$ we define $C= \p \bar C$ to be the restriction $\bar
C|_{M}$ and we denote by $\mathcal{C}(W,M)$ the space $\{(\bar C, C)
\ | \ \bar{C} \in \mathcal{C}(W), \ C=\p \bar C\ \}$.  We also
consider a normalized conformal class $\bar C^{(0)}= \{ \ \bar g \in
\bar C \ | \ H_{\bar g}\equiv 0 \ \}$. It is easy to observe that the
normalized class $\bar C^{(0)}\subset \bar C$ is always non-empty, and
there is a natural bijection between the spaces $\mathcal{C}(W,M)$ and
$\mathcal{C}^{(0)}(W,M) = \{ (\bar C^{(0)},C) \ | \ \bar C \in
\mathcal{C}(W), \ C=\p \bar C\ \}$ (see \cite[(1.4)]{Escobar92}).

Let $\Riem(W)$ stand for the space of Riemannian metrics on $W$.  
Fix a conformal class $C\in \mathcal{C}(M)$, we need the following
spaces of metrics:
\begin{equation*}
\begin{array}{c}
  \Riem_C(W,M)  =  \{ \bar g \in \Riem(W) \ | \ \p [\bar g] = C \ \}, \ \ \
  \Riem_C^{(0)}(W,M)  = \{ \bar g \in  \Riem_C(W,M) \ |
  \ H_{\bar g}\equiv 0 \ \}.
\end{array} 
\end{equation*}
Consider the normalized Einstein-Hilbert functional $I_C:
\Riem_C^{(0)}(W,M) \to \R$ given by
\begin{equation*}
  I_C(\bar g) =
  \frac{\int_W R_{\bar g}d\sigma_{\bar g}}{\Vol_{\bar g}(W)^{\frac{n-2}{n}}}
\end{equation*}  
where $R_{\bar g}$ is the scalar curvature and $d\sigma_{\bar g}$ is
the volume element.  Similarly to the case of closed manifolds, the
Einstein-Hilbert functional $I_C$ is not bounded for any manifold of
dimension $\geq 3$ with any fixed conformal class on the boundary.

Denote by $\mathrm{Crit}(I_C)$ the set of \emph{critical metrics} of the
functional $I_C$.  It is well-known that the set $\mathrm{Crit}(I_C)
\subset \Riem_C^{(0)}(W,M)$ coincides with the set of Einstein metrics
$\bar g$ on $W$ such that $\p[\bar g]=C$ and the mean curvature
$H_{\bar g}\equiv 0$ on $M$, see \cite{AB06}.

Fix $(\bar C, C)\in \mathcal{C}(W,M)$, the \emph{relative Yamabe
  constant} $Y_{\bar C}(W, M;C)$ is defined as
\begin{equation*}
  Y_{\bar C}(W, M;C)= \inf_{\bar g\in\bar C^{(0)}} I_C(\bar g).
\end{equation*} 
The functional $I_C|_{\bar C^{(0)}}: \bar g \mapsto I(\bar g)$ from
restriction is called the \emph{Yamabe functional}.

Let $\nu$ be the outward unit normal vector along the boundary $M$.
Fix a $\bar g \in \bar{C}^{(0)}$, then any metric in $\bar C^{(0)}$
can be written as $u^{\frac{4}{n-2}}\bar g$ where $u\in
C^{\infty}_+(W)$ is a smooth positive function satisfying
$\partial_{\nu} u \equiv 0 $ on $M$, see \cite[(1.4)]{Escobar92}.

Hence the relative Yamabe constant $Y_{[\bar
    g]}(W, M;[g])$ can be written as
\begin{equation}\label{eq rel Yamabe functional-b}
  Y_{[\bar g]}(W, M;[g])= \inf_{\small{\begin{array}{c}u \in
        C^{\infty}_+(W )\\ \partial_{\nu}u\equiv 0 \ \text{on
          $M$}\end{array}}} \frac{\int_{W}\left(\frac{4(n-1)}{n-2}
    \left\vert \nabla_{\bar g} u \right\vert^{2}+R_{\bar g}
    u^{2}\right) d\sigma _{\bar g}} {\left( \int_{W}
    u^{\frac{2n}{n-2}} d \sigma_{\bar g}\right)^{\frac{n-2}{n}}},
\end{equation} 
where $\nabla_{\bar g}$ is the Riemannian connection of the metric $\bar g$.  

The Euler-Lagrange equation for any
minimizer $u$ of the functional (\ref{eq rel Yamabe functional-b}) is
\begin{equation}\label{eq Yamabe Euler-Lag-b}
\begin{array}{cl}
\displaystyle
  -\frac{4\left( n-1\right) }{n-2}\Delta_{\bar g} u +R_{\bar g} u = Y_{[\bar
      g]}(W, M;[g]) u^{\frac{n+2}{n-2}}  & \quad \text{ on } W,  \\
\partial_{\nu} u \equiv 0  & \quad \text{
    on }  M.  
\end{array}  
\end{equation}
We assume that the  minimizer $u$ is normalized as
\begin{align}
   \int_{W} u^{\frac{2n}{n-2}}d\sigma_{\bar g}
  =1. \label{eq Yam vol 1} 
\end{align}
It is well-known that there exists a solution of (\ref{eq Yamabe
  Euler-Lag-b}) in a generic case due to Escobar, see \cite{Escobar92}
for details. Actually Escobar's results have been generalized in many
ways, for example, see \cite{Brendle14,Marcus06}.  It is also
well-known that the normalized solution $u$ is unique if $Y_{[\bar
    g]}(W, M;[g])\leq 0$ and there are examples of multiple solutions
if $Y_{[\bar g]}(W, M;[g])>0$.  For a normalized minimizer $u$, the
metric $u^{\frac{4}{n-2}}\bar g$ has volume one, constant scalar
curvature $Y_{[\bar g]}(W, M;[g])$, and zero mean curvature along the
boundary.  The metric is called a \emph{Yamabe metric}.

There is a remarkable property of the relative Yamabe constant $Y_{\bar C}(W,
M;C)$, namely, we have the inequality $ Y_{\bar C}(W, M;C)\leq Y_{\bar
  C_{st}}(D^n,S^{n-1};C_{st}), $ where $\bar C_{st}$ is conformal
classes of the standard metric on the hemisphere $D^n$ with totally
geodesic boundary $S^{n-1}$. Furthermore, the equality holds if and
only if the pair $((W,\bar C), (M,C))$ is conformally equivalent to
$((D^n,\bar C_{st}),(S^{n-1}, C_{st}))$.  This leads to the definition
of the \emph{Yamabe invariant} $Y(W, M;C)$:
\begin{equation}\label{Yamabe-inv}
  Y(W, M; C)= \sup_{\bar C, \, \p \bar C = C}
      Y_{\bar C}(W, M;C).
\end{equation} 
\begin{remark}
Assume that a conformal metric $\tilde g = \tilde
u^{\frac{4}{n-2}}\bar g$ has zero mean curvature on $M$ and has
constant scalar curvature $R_{\tilde g} = \tilde R$. Then the function
$\tilde u$ satisfies the Euler-Lagrange equations (\ref{eq Yamabe
  Euler-Lag-b}), where the relative Yamabe constant $Y_{[\bar g]}(W,
M;[g])$ is replaced by $\tilde R$. Hence the metric $\tilde
u^{\frac{4}{n-2}}\bar g$ is a critical metric, but not necessarily a
minimum of the functional $I_C|_{\bar C^{(0)}}$.
\end{remark}  

\subsection{The Result}
We consider the Ricci flow with the following boundary conditions:
\begin{equation}\label{RF-1} 
  \p_t \bar g_t= - 2 \Ric_{\bar g_t} \quad \text{ with mean curvature } 
   H_{\bar g_t} \equiv 0 
  \text{ and } \bar g_t|_M \in [\bar{g}_{0}|_M] =C_0. 
\end{equation}
Here is our main result:
\begin{thmA} \label{thm main}
Let $\bar{g}_t$ be the solution of Ricci flow {\rm (\ref{RF-1})} on $W$ with
initial metric $\bar g_{0}$ being a Yamabe metric.  Assume that
  there is a $C^{1}$-family of positive smooth functions $u(t)$,
  $t\in[0,\epsilon)$ for some $\epsilon>0$, with $u(0)=1$ such that the
    metric $u(t)^{\frac {4}{n-2}}\bar g_t$ is a Yamabe
    metric (with unit volume and constant scalar curvature $Y_{[\bar
        g_t]}(W, M;C_0)$).  Then $\left.  \frac{d} {dt}\right\vert
    _{t=0}Y_{[\bar g_t]}(W, M;C_0) \geq 0$ and the equality holds if and
    only if $\bar g_{0}$ is an Einstein metric. 
\end{thmA}
In fact, we will compute the evolution equation for the Yamabe
constant and the corresponding sub-critical constant under Ricci flow
(\ref{RF-1}), see Theorem B below. If the initial metric $\bar g_0$ is
a Yamabe metric, then the formula is very simple:
\begin{equation}\label{Yam}
  \left.  \frac{d}{dt}\right\vert _{t=0}Y_{[\bar g_t]}(W,M;C_0)
  =2 \int_{W}
  |\mathrm{Ric}^{0}_{\bar g_0}|^2 d\sigma_{\bar g_0} \geq 0,
\end{equation}
where $\mathrm{Ric}^{0}_{\bar g_0}$ is the norm of the traceless Ricci tensor.
The proof is given at the end of \S \ref{sec 2 proofs}.

\subsection{Ricci Flow on manifolds with boundary}
Since we have assumed the existence of the Ricci flow {\rm
  (\ref{RF-1})} on manifolds with boundary in Theorem A, here we
briefly review some results due to Gianniotis \cite{Gi16} on the
initial value problem of {\rm (\ref{RF-1})}.

Let $(W,\bar g^{(0)})$
be a Riemannian manifold with the mean curvature $H_{\bar
  g^{(0)}}\equiv 0$ along $M$. Choose conformal class
$C_{0}=[\bar{g}^{(0)}|_M]$ on boundary $M$.  According to
\cite[Theorem 1.2]{Gi16}, there exists a short-time solution
$\bar{g}_t$ of (\ref{RF-1}) in the space
\begin{equation*}
  C^{\infty}(W \times (0,T])    \cap C^{1+\alpha, (1+\alpha)/2}(W \times [0,T])
\end{equation*}
with the property that $\bar{g}_t$ converges to the initial metric
$\bar g^{(0)}$ in the $C^{1+\alpha}(W)$-Cheeger-Gromov topology and
$C^{\infty}$-topology away from the boundary $M$ as $t \rightarrow
0^+$, (i.e. it converges to $\bar g^{(0)}$
up to diffeomorphism). 

To get better regularity of the short-time solution $\bar{g}_t$ of
(\ref{RF-1}), the initial metric $\bar g^{(0)}$ has to satisfy the
following compatibility condition, namely,
\begin{equation} \label{eq h_1 order compat}
  (\operatorname{Ric}_{\bar{g}^{(0)}})^T =
  \tilde{f} \cdot (\bar{g}^{(0)})^T,   
\end{equation}
see \cite[Theorem 1.1]{Gi16}. Here $\tilde{f}$ is a smooth positive
function and $S^T$ denotes the tangential to the boundary part of the
tensor $S$, see \cite[p.314]{Gi16}.  Then the solution $\bar{g}_t$ is
in $C^{2+\alpha, (2+\alpha)/2}(W \times [0,T])$ and converges to the
initial metric $\bar g^{(0)}$ in $C^{2+\alpha}(W)$ (\cite[Theorem 4.2]{Gi16}).

From above discussion we conclude that the  relative Yamabe constant $Y_{[\bar
    g_t]}(W, M;[\bar{g}^{(0)}|_M])$ is differentiable for
$t>0$. Furthermore, if the compatibility condition (\ref{eq h_1 order
  compat}) is satisfied, then  relative Yamabe constant $Y_{[\bar g_t]}(W,
M;[\bar{g}^{(0)}|_M])$ is continuous at $t=0$.
\begin{remark}
Even though the boundary conditions given in (\ref{RF-1}) are not as
general as the ones studied by Gianniotis \cite{Gi16}, they define the
Ricci flow on the space $\Riem_{C}^{(0)}(W,M)$. Furthermore, the space
$\Riem_{C}^{(0)}(W,M)$ is suitable for defining a relative version of
the Perelman's functional to provide a gradient flow equivalent to the
Ricci flow (\ref{RF-1}).
\end{remark}
\subsection{Acknowledgement}
P.L.'s research is partially supported by Simons Foundation through
Collaboration Grant 229727 and 581101. We thank K. Akutagawa
and S. Hamanake for their help explaining to us a relevant Koiso's
decomposition type result for manifolds with boundary.

\section{Proofs} \label{sec 2 proofs}
\subsection{Subcritical Yamabe problem}
In order to analyze the Yamabe problem, it is important to consider
the \emph{sub-critical regularization} of the Yamabe functional, namely,
the functional $\mathsf{Y}_p:   \Riem^{(0)}(W) =
\{  \bar g \in \Riem(W) \ | \ H_{\bar g}\equiv 0 \} \rightarrow \mathbb{R}$ defined by
\begin{equation*}
\mathsf{Y}_p(W,\bar g):= \inf_{\small{\begin{array}{c}u \in C^{\infty}_+(W )\\
      \partial_{\nu}u\equiv 0 \ \text{on $M$}\end{array}}}
  \frac{\int_{W}\left(\frac{4(n-1)}{n-2}
    \left\vert \nabla_{\bar g}
    u \right\vert^{2}+R_{\bar g} u^{2}\right)  d\sigma _{\bar g}}
{\left( \int_{W} u^{p+1} d \sigma_{\bar g}\right)^{\frac{2}{p+1}}},
\end{equation*}  
for $p\in [1,\frac{n+2}{n-2})$. 
Note that the Yamabe
constant $Y_{[\bar g]}(W, M;[\bar{g}|_M])$ equals to the constant
$\mathsf{Y}_{p}(W,\bar g)$ if $p$ attains the critical value
$p=\frac{n+2}{n-2}$.

Clearly, the corresponding  Euler-Lagrange equation of functional $\mathsf{Y}_p$  is 
\begin{equation}\label{eq Yamabe Euler-Lag-b-p}
\begin{array}{cl}
\displaystyle
-\frac{4\left( n-1\right) }{n-2}\Delta_{\bar g}
u +R_{\bar g} u = \mathsf{Y}_{p}(W,\bar g)
 u^{p}  & \quad \text{ on } W,  \\
\partial_{\nu} u \equiv 0  & \quad \text{ on }  M.  
\end{array}  
\end{equation}
Again, we assume the following normalization condition:
\begin{align}
   \int_{W} u^{p+1}d\sigma_{\bar g}
  =1. \label{p-norm-b} 
\end{align}
\begin{remark}
It should be noted that the existence of solution $u$ of (\ref{eq
  Yamabe Euler-Lag-b-p}) and (\ref{p-norm-b}) follows from the
direct method in the calculus of variation because $p$ is sub-critical
exponent.  In that case the constants $\mathsf{Y}_{p}(W,\bar g)$ have
the same sign as the  relative Yamabe constant $Y_{[\bar g]}(W, M;[\bar{g}|_M])$,
however, the value of $\mathsf{Y}_{p}(W,\bar g)$ depends on the metric
$\bar g$, not only on the conformal class $[\bar g]$. In particular,
if $p=1$, $\mathsf{Y}_{p}(W,\bar g)$ coincides with the principal
eigenvalue of the conformal Laplacian with minimal boundary condition,
see \cite{Escobar92} for details. 
\end{remark}
\subsection{Evolution equations for the constant $\mathsf{Y}_p(W,\bar g)$}
Here is our main technical result:
\begin{thmB} \label{prop 1} 
Let $\bar g_t$ be a solution of the Ricci flow {\rm (\ref{RF-1})} on $W$ for $t\in [0,T)$.
Denote $g_t = \bar g_t|_M$. Given $p\in (1,\frac{n+2}{n-2}]$, assume
that there is a $C^{1}$-family of smooth positive functions $u(t)$,
$t \in [0,T)$, which satisfy
\begin{align}
& -\frac{4 (n-1)  }{n-2}\Delta_{\bar{g}_ t}u (
t)  +R_{\bar{g}_ t}u (  t)     =\mathsf{Y}_p(W,\bar g_t) u(  t)  ^{p} 
\quad \text{ on } W, \label{eq 1p} \\
& \ \ \ \ \ \ \ \ \ \ \ \ \ \ \ \ \ \int_{W}u(  t)  ^{p+1}d\mu_{\bar{g}_ t}    =1,
\label{eq volume 1p}  \\
&
\ \ \ \ \ \ \ \ \ \ \ \ \ \ \ \ \ 
\partial_{\nu_t} u (t) =0 \quad \text{ on } M,   \label{eq u t normal} 
\end{align}
where $\nu_t$ is the outward  unit normal vector with respect to metric
$\bar{g}_t$.  Then

\begin{align}
 \frac{d}{dt}\mathsf{Y}_p(W,\bar g_t)  = & \int_{W}\left(
\frac{8 ( n-1 )  }{n-2} \overline{\Ric}^{0} ( \bar{\nabla}u, \bar{\nabla}u)+2\left\vert
\overline{\Ric}^{0}\right\vert ^{2}u^{2} \right)  d\sigma_{\bar{g}_ t} \nonumber \\
&  +\left(  \frac{2}{n}-\frac{p-1}{p+1}\right)  \int_{W}\left(  \frac{4(
n-1)  }{  n-2  } \bar{R} \left\vert \bar{\nabla} u \right\vert ^{2}
+\bar{R}^{2}u^{2}\right)  d\sigma_{\bar{g}_ t } \nonumber \\ 
& +\int_{W} \left (u^2 \bar{\Delta} \bar{R} - \frac{4(n-1)}{(p+1)(n-2)  } \bar{R} \bar{\Delta} 
u^{2} \right ) d \sigma_{\bar{g}_ t}  \nonumber \\
&   -\frac{8(n-1)}{n-2} \int_M \left ( 2 u \overline{\Ric}(\bar{\nabla} u, \nu_t)  +  u 
\partial_{\nu_t} h \right ) d \sigma_{g_t} ,  \label{eq deriv yam p}
\end{align}
where $u=u\left( t\right) $, $h=\frac{\partial u}{\partial t}$,  and
$\overline{\Ric}^{0}$, $\bar{\nabla}$, $\bar{\Delta}$, and $\bar{R}$ are the traceless Ricci tensor,
the Riemann connection, the Laplace-Beltrami operator, and the scalar
curvature of metric $\bar{g}_ t$, respectively.
\end{thmB}
\begin{proof}
We use a short-hand notation $\mathsf{Y}_p(t)=\mathsf{Y}_p(W,\bar
g(t))$.  First we note that
\begin{equation*}
\mathsf{Y}_p(t) = \int_{W}\left( \frac{4\left( n-1\right)
}{n-2}\left\vert \bar{\nabla} u\right\vert ^{2}+\bar R u^{2}\right)
d\sigma_{\bar g_t}.
\end{equation*}
We compute 
\begin{align}
\frac{d}{dt}\mathsf{Y}_p(t)  &  =\int_{W}\left(
\frac{8(  n-1)  }{n-2}\overline{\Ric}(\bar \nabla u, \bar \nabla u)+\frac{8 (
  n-1 )  }{n-2} \langle \bar \nabla u, \bar \nabla h \rangle \right)  d\sigma_{\bar g_t} \nonumber   \\
&  +\int_{W}\left(  \left(  \bar \Delta \bar R+
2\left\vert \overline{\operatorname*{Ric}}\right\vert
^{2}\right)  u^{2}+2\bar R u h\right)  d\sigma_{\bar g_t} \label{eq deriv yp const}\\
&  -\int_{W}\left(  \frac{4 (  n-1 )  }{n-2}\left\vert \bar \nabla
u\right\vert^{2}+\bar R u^{2}\right)  \bar R \
d\sigma_{\bar g_t} , \nonumber 
\end{align}
where we have used $\frac{\partial \left\vert \bar \nabla u\right\vert
  ^{2} }{\partial t} = 2\overline{\Ric} ( \bar \nabla u, \bar \nabla u) +2
\langle \bar \nabla u, \bar \nabla h \rangle$, $\frac{\partial \bar R }{\partial t}
= \bar \Delta \bar R + 2 \left\vert \overline{\operatorname*{Ric}} \right\vert
^{2}$ and $ \frac{\partial d\sigma_{\bar g_t} }{\partial t} = -\bar R\ 
d\sigma_{\bar g_t} $.

To eliminate the terms containing $h$ in formula (\ref{eq deriv yp
  const}), we use integration by parts to rewrite
\begin{align}
& \int_{W} \left ( \frac{8 (n-1)  }{n-2} \langle \bar \nabla u, \bar \nabla h \rangle
 +2\bar R u h \right )  d\sigma_{\bar g_t} \nonumber \\
=&\int_{W} \left ( -\frac{8(  n-1)  }{n-2}u\bar \Delta h+2\bar R u h
\right ) d\sigma_{\bar g_t} 
+ \frac{8(n-1)}{n-2}
\int_M u \partial_{\nu_t}h d \sigma_{g_t}. \label{eq integ parts 1}
\end{align}
Taking derivative $\frac{d}{dt}$ of both sides of the equation
(\ref{eq 1p}) and then multiplying the result by $2u$ as in
\cite[p.149]{Lu}, we get
\begin{align*}
-\frac{8(n-1)  }{n-2}u\bar \Delta h+2 \bar{R}uh  &  =\frac{16 (
  n-1 )  }{n-2}u \langle\overline{\Ric}, \bar \nabla \bar
\nabla u \rangle -2\left(  \bar \Delta \bar  R+2\left\vert
\overline{\operatorname*{Ric}}\right\vert ^{2}\right)  u^{2}\\
&  +2 \left (  \frac{d}{dt}\mathsf{Y}_{p}(t)  \right )
u^{p+1}+2p\mathsf{Y}_{p}(t)  u^{p}h.
\end{align*}
Plugging this formula into equation (\ref{eq integ parts 1}) and then 
plugging the resulting equation into the formula (\ref{eq deriv yp const}) for 
$\frac{d}{dt}\mathsf{Y}_{p}(t)$ to eliminate
the terms containing $h$, we have
\begin{align*}
  \frac{d}{dt}\mathsf{Y}_{p}(t) = & \int_{W}
  \left ( \frac{8 ( n-1 ) }{n-2} \overline{\Ric}( \bar \nabla u , \bar 
  \nabla u ) - \left( \bar \Delta R+2\left\vert
\overline{\operatorname*{Ric}}\right\vert ^{2}\right) u^{2} \right ) d
\sigma_{g_t}\\ & +\int_{W} \left ( \frac{16 ( n-1)
}{n-2}u \langle \overline{\Ric},\bar \nabla \bar \nabla u \rangle - \left(\frac{4 (n-1)
}{n-2}\left\vert \bar \nabla u\right\vert ^{2} +\bar R u^{2}\right) \bar{R} \right )
d \sigma_{\bar g_t}\\
& +\frac{2p}{p+1} \mathsf{Y}_{p}( t)
\int_{W}u^{p+1}\bar R d \sigma_{\bar g(t)} +2\frac{d}{dt}\mathsf{Y}_{p}( t)
+ \frac{8(n-1)}{n-2} \int_M u \partial_{\nu_t}h d
\sigma_{g_t},
\end{align*}
where we have used
\begin{equation*}
\int_{W}u^{p}hd\sigma_{\bar g_t} =\frac{1}{p+1}\int_{W}u^{p+1} \bar{R}
d \sigma_{\bar g_t}
\end{equation*}
which is obtained by taking derivative $\frac{d}{dt}$ of the
constraint (\ref{eq volume 1p}).

To simplify further the above formula for
$\frac{d}{dt}\mathsf{Y}_{p}(t)$, we compute using integration by parts
\begin{align*}
& \int_{W}u \langle \overline{\Ric}, \bar \nabla \bar \nabla u \rangle d \sigma_{\bar g_t} \\
  =& -\frac{1}{2}\int_{W}u \langle \bar \nabla \bar R, \bar \nabla u \rangle
  d \sigma_{\bar g_t} -\int_{W} \overline{\Ric} (\bar \nabla u,
  \bar \nabla u) d\sigma_{\bar g_t} +
  \int_M u \overline{\Ric} (\bar{ \nabla} u, \nu_t ) d \sigma_{g_t} \\
  = & \frac{1}{4}\int_{W}\bar R \bar \Delta u^{2}d \sigma_{\bar g_t} -
  \int_{W} \overline{\Ric} (\bar \nabla u, \bar \nabla u) d
\sigma_{g_t} + \int_M u \overline{\Ric} ( \bar{\nabla} u, \nu_t) d \sigma_{g_t},
\end{align*}
where we have used $\partial_{\nu_t}u =0$ to get the
last equality.  Hence we obtain
\begin{align}
\frac{d}{dt}\mathsf{Y}_{p}( t) = & \int_{W} \left (
\frac{8\left(  n-1\right)  }{n-2} \overline{\Ric}( \bar \nabla u, \bar \nabla u) + 
\left(  \bar \Delta \bar R+2
\left\vert \overline{\operatorname*{Ric}}\right\vert ^{2}\right)
u^{2} \right ) d \sigma_{\bar g_t} \nonumber \\
& + \int_{W} \left ( -\frac{4(n-1)  }{n-2} \bar R \bar \Delta u^{2} + \left (
\frac{4 (  n-1)  }{n-2}\left \vert \bar
\nabla u\right\vert ^{2}+\bar R u^{2}\right ) \bar{R}
 \right )  d \sigma_{\bar g_t} \label{eq tem 1}\\
& -\frac{2p}{p+1}\mathsf{Y}_{p}( t)  \int_{W}u^{p+1}
 \bar R d \sigma_{\bar g_t}   -\frac{16(n-1)}{n-2} \int_M u \overline{\Ric} ( \bar
 \nabla u, \nu_t) d  \sigma_{g_t}  \nonumber \\
& - \frac{8(n-1)}{n-2} \int_M u  \partial_{\nu_t}h d \sigma_{g_t}.
 \nonumber
\end{align}
Next we eliminate $\mathsf{Y}_{p}( t) $ from the
right-hand side of (\ref{eq tem 1}). Multiplying (\ref{eq 1p}) by $\bar{R}u$ and
integrating we get
\begin{align*}
\mathsf{Y}_{p}( t) \int_{W}u^{p+1}\bar R d \sigma_{\bar g_t}
= & \int_{W} \left ( - \frac{2(n-1)  }{n-2}\bar R \bar \Delta u^{2}+ \left(
\frac{4 (n-1)  }{n-2}\left\vert \bar \nabla u\right\vert ^{2}
+\bar R u^{2} \right)  \bar R \right ) d \sigma_{\bar g_t} ,
\end{align*}
hence
\begin{align*}
\frac{d}{dt}\mathsf{Y}_{p}\left(  t\right)   &  =\int_{W} \left ( 
\frac{8\left(  n-1\right)  }{n-2}\overline{\Ric} (\bar \nabla u, \bar \nabla u) + 
2\left\vert \overline{\operatorname*{Ric}}\right\vert ^{2}u^{2} \right )
d \sigma_{\bar g_t}\\
&  -\frac{p-1}{p+1}\int_{W}\left(  \frac{4\left(  n-1\right) }{
  n-2 }\left\vert \bar \nabla u\right\vert ^{2}+\bar Ru^{2}\right)  \bar R
d \sigma_{\bar g_t}  \\
&  +\int_{W} \left (u^2 \bar \Delta \bar R -
\frac{4(n-1)}{(p+1)(n-2)  }\bar R \bar \Delta u^{2}\right ) d \sigma_{\bar g_t}\\
&   -\frac{8(n-1)}{n-2} \int_M \left ( 2 u \overline{\Ric} (\nabla u, \nu_t) + u 
\partial_{\nu_t}h \right ) d \sigma_{g_t} .
\end{align*}
As calculated in \cite[p.151]{Lu} by using $\overline{\Ric}=
\overline{\Ric}^{0}+\frac{\bar{R}}{n} \bar g$ and $\left\vert
\overline{\Ric} \right\vert ^{2} =\left\vert \overline{\Ric}^0\right\vert ^{2} + 
\frac{\bar R^2}{n}$, we get (\ref{eq deriv yam p}) from the formula above.
\end{proof}
 \begin{remark}
  Concerning the assumption in Theorem A,
  that there are functions $u(t)$ such that the metrics $u \left(
  t\right) ^{\frac{4} {n-2}} \bar{g}_t $ are a $C^{1}$-family of
  smooth Yamabe metrics, there are two cases here.
\begin{enumerate}
\item[{Case 1:}] The relative Yamabe constant $Y_{[\bar
    g_0]}(W,M;[g_0])\leq 0$. Then there is a unique solution $u(0)$ of
  the Yamabe problem, and it could be shown that for small $t$ the
  Yamabe metrics $u(t)^{\frac{4} {n-2}}\bar{g}_t$ is a smooth in $t$
  family of metrics, this is a consequence of recent result due to
  S. Hamanake (\cite{AK}) who proved a relevant Koiso's decomposition type
  result for manifolds with boundary (cf. to the Koiso’s decomposition
  theorem \cite[Corollary 2.9]{Ko79} for closed manifolds).
\item[{Case 2:}] The  relative Yamabe constant $Y_{[\bar g_0]}(W,M;[g_0])>
  0$. Then, in general, the corresponding Yamabe metric is not
  unique. Thus in this case it is not clear whether there exists a
  $C^{1}$-family of smooth functions $u(t)$ which satisfy the
  assumption even for a short time.
\end{enumerate}
\end{remark}
\subsection{Boundary terms}
Below we compute the boundary terms appearing in (\ref{eq deriv yam
  p}) using local coordinates. Fix a boundary point $p \in M$ and a
time $t$, we choose local coordinates $(x^1, \cdots, x^n)$ on $W$
around $p$ such that $(\bar{g}_t)_{ij}(p)=\delta_{ij}$ and the outward
unit normal vector $\nu_t(p) = \partial_n =\frac{\partial}{\partial
  x^n}$.
 
{\bf Part A,  $\partial_t \nu_t$}. 
The outward unit  normal vector at time $\tilde{t}$ near $t$ can be written as
\[
\nu_{\tilde{t}} (p)= \frac{1}{b(\tilde{t})} \left (\sum_{i^{\prime}=1}^{n-1}a^{i^{\prime}}
(\tilde{t}) \partial_{i^{\prime}}  +\partial_n \right ),
\]
where $a^{i^{\prime}} (t) =0$ and 
\[
b(\tilde{t}) = \sqrt{(\bar{g}_{\tilde{t}})_{nn}(p ) + 2 \sum_{i^{\prime}=1}^{n-1}(\bar{g}_{\tilde{t}}
)_{i^{\prime}n}( p) a^{i^{\prime}}(\tilde{t}) + \sum_{i^{\prime},j^{\prime}=1}^{n-1}(\bar{g}_{
\tilde{t}})_{i^{\prime} j^{\prime}} (p) a^{i^{\prime}}(\tilde{t}) a^{j^{\prime}} (\tilde{t})    }.
\] 
From 
\[
0= b(\tilde{t}) \cdot  \bar{g}_{\tilde{t}}(\nu_{\tilde{t}}, \partial_{i^{\prime}}) (p) =
 \sum_{j^{\prime}=1}^{n-1} a^{j^{\prime}}(\tilde{t}) (\bar{g}_{\tilde{t}})_{i^{\prime}j^{\prime}} (p)+ 
(\bar{g}_{\tilde{t}})_{i^{\prime}n} (p)
\]
for each $i^{\prime}=1,2, \cdots, n-1$,  by taking the derivative
$\left . \frac{d}{d\tilde{t}} \right |_{\tilde{t} =t}$ of the equality above  we have
\[
 \sum_{i^{\prime}=1}^{n-1} \left ( \left . \frac{d a^{j^{\prime}}}{d \tilde{t}}
  \right |_{\tilde{t} =t}  \delta_{i^{\prime}j^{\prime}} 
 + a^{j^{\prime}}(t) (-2(\operatorname{
Ric}_{\bar{g}_t})_{i^{\prime} j^{\prime}} (p)) \right ) -2
(\operatorname{Ric}_{\bar{g}_t})_{i^{\prime} n}(p) =0.
\]
 Hence 
 \[
 \left . \frac{d a^{i^{\prime}}}{d \tilde{t}}  \right |_{\tilde{t} =t}  =2(\operatorname{Ric}
 _{\bar{g}_t})_{i^{\prime} n}(p), \qquad 
 \left . \frac{d b}{d\tilde{t}}  \right |_{\tilde{t} =t} =-(\operatorname{Ric}_{\bar{g}_t})_{nn}
 \]
 and
 \begin{equation} \label{eq deriv out normal}
   \partial_t \nu_t (p) =2 \sum_{i^{\prime}=1}^{n-1}
   (\operatorname{Ric}_{\bar{g}_t})_{i^{\prime}
  \nu_t} (p) \partial_i  + (\operatorname{Ric}_{\bar{g}_t})_{ \nu_t  \nu_t} (p) \partial_{ \nu_t}.
 \end{equation}
\vspace{2mm}

{\bf Part B, $\partial_{\nu_t}h$}. Write $\nu_t = \nu^i_t \partial_i$. 
   From $ \partial_n u  = \partial_{\nu_t}u =0$ on $M$, 
  we have by (\ref{eq deriv out normal})
  \begin{align*}
 0= & \ \ \partial_t (\nu^i_t \partial_i u)   =
 (\partial_t \nu^i_t ) \partial_i u + \nu^i_t  \partial_t  \partial_i u \\
 = & \ \ 2  \sum_{i^{\prime}=1}^{n-1}  (\operatorname{Ric}_{\bar{g}_t})_{i^{\prime}n} 
 \partial_{i^{\prime}} u + (\operatorname{Ric}_{\bar{g}_t})_{nn} \partial_n u 
 + \nu^i_t \partial_i  h.
  \end{align*}
We have established
  \begin{equation}
 \partial_{\nu_t}h  = -2 \sum_{i^{\prime}=1}^{n-1}  (\operatorname{Ric}_{\bar{g}_t})
 _{i^{\prime} \nu_t} \partial_{i^{\prime}} u. \label{eq evol eq normal der h}
 \end{equation}
\vspace{2mm}

{\bf Part C, $\overline{\Ric}( \bar{\nabla} u, \nu )$}.   This term is
\begin{align}
\bar{R}_{ij} \bar{\nabla}_i u \nu_j =& \sum_{i^{\prime}=1}^{n-1} \bar{R}_{i^{\prime}n} 
\partial_{i^{\prime}}  u + \bar{R}_{nn} \bar{\nabla}_n u  \notag \\
= & \sum_{i^{\prime}=1}^{n-1}  (\operatorname{Ric}_{\bar{g}_t})_{i^{\prime} \nu_t} 
\partial_{i^{\prime}} u  .  \label{eq ricci boundary  term}
\end{align}

\vskip .2cm

Finally we give a proof of Theorem A.  When $p=\frac{n+2}{n-2}$ in Theorem
\ref{prop 1}, we have that $\mathsf{Y}_{p}(t)=Y_{[\bar
    g_t]}(W,M;[g_t])$ and
\begin{align}
 \frac{d}{dt}Y_{[\bar g_t]}(W,M;[g_t])  = & \int_{W}\left(
\frac{8 ( n-1 )  }{n-2} \overline{\Ric}^{0} ( \bar{\nabla} u, \bar{\nabla} u) +2\left\vert
\overline{\Ric}^{0}\right\vert ^{2}u^{2} \right)  d\sigma_{\bar{g}_t} \nonumber \\
& +\int_{W} \left (u^2 \bar{\Delta} \bar{R} - \frac{2(n-1)}{n  }
\bar{R} \bar{\Delta} 
u^{2} \right ) d \sigma_{\bar{g}_ t}  \label{eq deriv yam}\\
&   -\frac{8(n-1)}{n-2} \int_M \left ( 2 u \overline{\Ric} (\nabla u, \nu)  +  u 
\partial_{\nu_t}h \right ) d \sigma_{g_t} \nonumber .  
\end{align}

Note that if the initial metric $\bar g_0$ is a Yamabe metric with
constant scalar curvature $R_{\bar g_0}= Y_{[\bar g_0]}(W,M;[g_0])$,
then $u(0)=1$ and $\left .  \partial_{\nu_t} h \right
|_{t=0} =0$ by (\ref{eq evol eq normal der h}), thus the formula
(\ref{eq deriv yam}) evidently reduces to (\ref{Yam}).


\end{document}